\documentclass[journal]{article}

\usepackage{graphicx,url}
\usepackage{dcolumn}
\usepackage{bm}
\usepackage{amsmath,amsfonts,amssymb}

\newtheorem{definition}{\noindent{\it Definition}}[section]
\newtheorem{theorem}{\noindent{\it Theorem}}[section]

\newtheorem{remark}[theorem]{\noindent{\it Remark}}

\newenvironment{proof}{\noindent{\it Proof:}}{$\hfill$ $\Box$\\ }

\begin{document}

\title{Asymptotically Good Convolutional Codes}

\author{Giuliano G. La Guardia
\thanks{Giuliano Gadioli La Guardia is with Department of Mathematics and Statistics,
State University of Ponta Grossa (UEPG), 84030-900, Ponta Grossa,
PR, Brazil. }}

\maketitle

\begin{abstract}
In this paper, we construct new sequences of asymptotically good
convolutional codes (AGCC's). These sequences are obtained from
sequences of transitive, self-orthogonal and self-dual algebraic
geometry (AG) codes attaining the Tsfasman-Vladut-Zink bound.
Furthermore, by applying the techniques of expanding, extending,
puncturing, direct sum, the $\langle u | u + v\rangle$ construction
and the product code construction to these block codes, we construct
more new sequences of asymptotically good convolutional codes.
Additionally, we show that the new constructions presented here also
hold when applied to al l sequences of asymptotically good block
codes where ${\lim}_{j \rightarrow \infty} k_{j}/n_{j}$ and
${\lim}_{j \rightarrow \infty} d_j /n_{j}$ exist.
\end{abstract}


\section{Introduction}\label{Intro}

In his seminal work \cite{Goppa:1972}, Goppa  introduced the
well-known class of algebraic geometry (AG) codes. It is well known
that the AG codes are asymptotically good; because of this fact,
coding theorists have utilized this structure in order to obtain
(maximal) curves over which more families of asymptotically good
codes can be constructed (see
\cite{Goppa:1972,Goppa:1981,Henning:1995,Henning:1996,Henning:2005}).
More information with respect to the class of AG codes can be found
in \cite{Henning:2009}.

Concerning the investigation and development of theory of
convolutional codes, much effort has been paid
\cite{Forney:1970,Piret:1988,Rosenthal:1999,Rosenthal:2001,Hutchinson:2005,Gluesing:2006,Schmale:2006,Iglesias:2009,Kuijper:2009,Napp:2017}.
More specifically, constructions of convolutional codes with good or
even optimal parameters (for instance, maximum-distance-separable
(MDS) codes, i.e., codes attaining the generalized Singleton bound
\cite{Rosenthal:1999}) are of great interest for several researchers
\cite{Rosenthal:1999,Rosenthal:2001,Hutchinson:2005,Gluesing:2006,LaGuardia:2014,LaGuardia:2014A,LaGuardia:2016}.

In this context, constructions of families of good or asymptotically
good convolutional codes (AGCC) are also of great importance in the
literature
\cite{Justesen:1973,Papadimitriou:2001,Sridharan:2007,Costello:2008,Costello:2010,Uchikawa:2011,Mitchell:2013,Liu:2014,Mitchell:2015,Bocharova:2016}.
In \cite{Justesen:1973}, Justesen constructed AGCC generated by
generator polynomials of cyclic codes. In \cite{Papadimitriou:2001},
the authors utilized a different approach to design good
convolutional codes. More precisely, the authors assume no
restrictions in the rate of the code, but only restrictions in the
constraint length. Using this approach, they found families of
AGCC's with applications in CDMA systems. In \cite{Sridharan:2007},
the authors introduced an ensemble of $(J; K)$-regular LDPC
convolutional codes and they presented lower bounds on the free distance
of such codes. After this, they showed that the ratio between the free distance and the
constraint length is several times larger than the ratio of minimum
distance for Gallager's ensemble of $(J;K)$-regular
LDPC block codes. In \cite{Costello:2008},
the authors computed a lower bound on the free distance for several
ensembles of asymptotically good protograph-based low-density
parity-check (LDPC) convolutional codes. They utilized ensembles of
LDPC convolutional codes (which were introduced by Felstrom and
Zigangirov in \cite{Felstrom:1999}) derived from a protograph-based
ensemble of LDPC block codes to obtain asymptotically good,
periodically time-varying LDPC convolutional code ensembles, having
the property that the free distance grows linearly with constraint
length. In \cite{Costello:2010}, the authors performed an iterative
decoding threshold analysis of LDPC block code ensembles formed by
certain classes of LDPC convolutional codes. These ensembles were
shown to have minimum distance growing linearly with block length
and their thresholds approach the Shannon limit as the termination
factor tends to infinity. In \cite{Uchikawa:2011}, Uchikawa \emph{et
al.} generalized the results shown in \cite{Felstrom:1999} to
non-binary LDPC convolutional codes. They also investigated its
rate-compatibility. In particular, they modified the construction
method proposed in \cite{Felstrom:1999}, to construct a non-binary
$(2,4)$-regular LDPC convolutional code. Applying numerical
simulations they showed that non-binary rate $1/2$ LDPC
convolutional codes binary LDPC convolutional codes with comparable
constraint bit length. Mitchell \emph{et al.} \cite{Mitchell:2013}
showed that several ensembles of regular and irregular LDPC convolutional
codes derived from protograph-based LDPC block codes
are asymptotically good. Furthermore, they investigated
the trapping set (see the definition in \cite{MacKay:2003})
of such class of codes. Mu \emph{et al.} \cite{Liu:2014} constructed
time-varying convolutional low-density parity check (LDPC) codes
derived from block LDPC codes based on improved progressive edge
growth (PEG) method (see \cite{Chen:2005}). Different from the
conventional PEG algorithm, the parity-check matrix is initialized
by inserting certain patterns. Applying simulation results the
authors showed that the new convolutional codes perform well over
additive white Gaussian noise(AWGN) channels. Mitchell \emph{ et al.}
\cite{Mitchell:2015} investigated relationship
between the minimum distance growth rate of the
spatially coupled LDPC block codes (SC-LDPC-BC) ensemble and the free distance growth rate of
the associated spatially coupled LDPC convolutional codes (SC-LDPC-CC) ensemble.
They showed that the minimum distance growth rates converge to a bound on the
free distance growth rate of the unterminated SC-LDPC-CC ensemble. Bocharova \emph{et
al.} \cite{Bocharova:2016} proposed an interesting search for good
convolutional LDPC codes over binary as well as non-binary alphabets
by means of three algorithms. They presented examples of codes with
bi-diagonal structure of the corresponding parity-check matrix,
which preserves low encoding complexity.

In great part of the works mentioned above, the authors considered
time-varying convolutional codes, whereas, in our case, we construct
convolutional codes whose (reduced basic) generator matrices do not
depend on the time.

In \cite{Rosenthal:1999}, the authors introduced the generalized Singleton bound: if $C$ is an
$(n, k, \gamma ; m, d_{f} {)}_{q}$ convolutional code then
$d_{f}\leq (n-k)[ \lfloor \gamma/k \rfloor +1 ] + \gamma +1$, where
$n$ denotes the length, $k$ is the dimension, $\gamma$ is the degree
of the code, $m = {\max}_{1\leq i\leq k}\{{\gamma}_i\}$ is the
\emph{memory} and $d_{f}=$wt$(V)=\min \{wt({\bf v}(D)) \mid {\bf
v}(D) \in V, {\bf v}(D)\neq 0 \}$ denotes the \emph{free distance}
of the code. This upper bound clearly generalizes the Singleton
bound ($\gamma = 0$ for block codes).

Based on the generalized Singleton bound, given a sequence
${(V_{j})}_{j\geq 0}$ of convolutional codes with parameters $(n_j,
k_j, {\gamma}_j ; m_j, {(d_f)}_{j} {)}_{q}$, we present the first
contribution of this paper, that is, we introduce the quantities
$s_j := (n_j -k_j)[ \lfloor {\gamma}_j /k_j \rfloor +1 ] +
{\gamma}_j +1$ and $r_j := \max \{n_j , s_j \}$. These two
quantities allow us to define more naturally the concept of AGCC in
the following sense (see Definition~\ref{defmain1}). Considering the
sequence ${(V_{j})}_{j\geq 0}$ of convolutional codes given above,
if ${\limsup}_{j \rightarrow \infty}$ $k_{j}/n_{j} > 0$ and if
${\limsup}_{j \rightarrow \infty} {(d_f)}_{j}/r_j > 0$ hold, where
$s_j := (n_j -k_j)[ \lfloor {\gamma}_j /k_j \rfloor +1 ] +
{\gamma}_j +1$ and $r_j := \max \{n_j , s_j \}$, then one says that
the sequence is asymptotically good. On the other hand one says that
a sequence ${(C_{j})}_{j\geq 0}$ of linear block codes over ${
\mathbb F}_{q}$ with parameters ${[n_j , k_j , d_j ]}_{q}$ is
asymptotically good if $n_j \rightarrow \infty$ as $j \rightarrow
\infty$ and if ${\limsup}_{j \rightarrow \infty} k_{j}/n_{j}> 0 $
and ${\limsup}_{j \rightarrow \infty} d_{j}/n_{j} > 0$ are true.
Thus, these two definitions are analogous in the sense that they
consider a comparison between length and dimension as well as they
provide a comparison between the free distance and the quantity
$r_j$, which is the maximum between $n_j$ and the Singleton bound
$s_j$ (note that the consideration of $s_j$ in the case of
convolutional codes is natural, because $s_j$ could be greater than
$n_j$).

As the second contribution, starting from transitive AG codes in
\cite{Henning:2005}, we construct new families of convolutional
codes having reduced basic generator matrices by utilizing Piret's
technique \cite{Piret:1972}. After this, we show that these new
families of convolutional codes are asymptotically good. An
advantage of our method is that it is performed algebraically. More
precisely, given a family of asymptotically good block codes with
given parameters (in particular, the transitive codes shown in this
paper), our technique allows us to construct the corresponding
family of asymptotically good convolutional codes and also to
compute the exact parameters of these codes (except the free
distance, where a lower bound is given). Additionally, we do not
utilize algorithms nor computational search to this end.

The third contribution of this work is based on the second one:
starting from techniques of code expansion, extension, direct sum,
the $\langle u|u+v\rangle$ construction, puncturing and direct
product construction applied to such AG codes, we carefully
construct sequences of convolutional codes with reduced basic
generator matrices, after showing that such families are good
asymptotically. Moreover, we show that our proposed constructions
also hold when applied to all sequences of asymptotically good block
codes where ${\lim}_{j \rightarrow \infty} k_{j}/n_{j}$ and
${\lim}_{j \rightarrow \infty} d_j /n_{j}$ exist.

Another well-known class of asymptotically good codes is the class
of random binary and non-binary convolutional codes with
low-density parity-check matrices (see \cite{Uchikawa:2011,Liu:2014}).
Since such class of convolutional codes is asymptotically good,
it is interesting to know how to relate the free distances
of these codes with the ones presented in this paper.
However, because random convolutional LDPC codes
are often constructed by means of algorithms
or they are found by computational search, and the new
asymptotically good convolutional codes
presented here are constructed algebraically, without using algorithms or even
without using computational search, it is difficult to compare their free distances
for fixed block lengths. In other words, due to the big difference between both
construction methods we cannot perform a comparison among such codes.

The paper is organized as follows. In Section~\ref{II}, we review
the concepts on convolutional codes. In Section~\ref{III}, we
establish the contributions of this work, i.e., constructions of new
sequences of asymptotically good convolutional codes. Finally, in
Section~\ref{IV}, the final considerations are drawn.

\section{Review of Convolutional Codes}\label{II}

In this section we present a brief review of convolutional codes.
For more details we refer the reader to
\cite{Forney:1970,Johannesson:1999,Huffman:2003}.

Throughout this paper, $q$ is a prime power and ${ \mathbb F}_{q}$
is the finite field with $q$ elements. Recall that a polynomial
encoder matrix $G(D) \in { \mathbb F}_{q}{[D]}^{k \times n}$ is
called \emph{basic} if $G(D)$ has a polynomial right inverse. A
basic generator matrix is called \emph{reduced} (or minimal
\cite{Rosenthal:2001,Huffman:2003}) if the overall constraint length
$\gamma =\displaystyle\sum_{i=1}^{k} {\gamma}_i$, where ${\gamma}_i
= {\max}_{1\leq j \leq n} \{ \deg g_{ij} \}$, has the smallest value
among all basic generator matrices (in this case the overall
constraint length $\gamma$ will be called the \emph{degree} of the
resulting code).

\begin{definition}\cite{Johannesson:1999}
A rate $k/n$ \emph{convolutional code} $V$ with parameters $(n, k,
\gamma ;$ $m, d_{f} {)}_{q}$ is a submodule of ${ \mathbb F}_{q}
{[D]}^{n}$ generated by a reduced basic matrix $G(D)=(g_{ij}) \in {
\mathbb F}_q {[D]}^{k \times n}$, i.e. $V = \{ {\bf u}(D)G(D) | {\bf
u}(D)\in { \mathbb F}_{q} {[D]}^{k} \}$, where $n$ is the length,
$k$ is the dimension, $\gamma$ is the degree, $m = {\max}_{1\leq
i\leq k}\{{\gamma}_i\}$ is the \emph{memory} and $d_{f}=$wt$(V)=\min
\{wt({\bf v}(D)) \mid {\bf v}(D) \in V, {\bf v}(D)\neq 0 \}$ is the
\emph{free distance} of the code.
\end{definition}

In the above definition, the \emph{weight} of an element ${\bf
v}(D)\in { \mathbb F}_{q} {[D]}^{n}$ is defined as wt$({\bf
v}(D))=\displaystyle\sum_{i=1}^{n}$wt$(v_i(D))$, where wt$(v_i(D))$
is the number of nonzero coefficients of $v_{i}(D)$. Let us consider
the field of Laurent series ${ \mathbb F}_{q}((D))$, whose elements
are given by ${\bf u}(D) = {\sum}_{i} u_i D^{i}$, where $u_i \in {
\mathbb F}_{q}$ and $u_i = 0$ for $i\leq r $, for some $r \in
\mathbb{Z}$. The weight of ${\bf u}(D)$ is defined as wt$({\bf
u}(D)) = {\sum}_{\mathbb{Z}}$wt$(u_i)$. A generator matrix $G(D)$ is
called \emph{catastrophic} if there exists a ${\bf u}{(D)}^{k}\in {
\mathbb F}_{q}{((D))}^{k}$ of infinite Hamming weight such that
${\bf u}{(D)}^{k}G(D)$ has finite Hamming weight. Since a basic
generator matrix is non-catastrophic, all the convolutional codes
constructed in this paper have non-catastrophic generator matrices.

We next recall how to construct a convolutional code derived from a
block code. This technique was presented first by
Piret~\cite{Piret:1972,Piret:1988} to binary codes, after
generalized by Aly \emph{et al.}~\cite{Aly:2007} to nonbinary
alphabets.

Let $C$ be an ${[n, k, d]}_{q}$ linear block code with parity check
matrix $H$. We split $H$ into $m+1$ disjoint submatrices $H_i$ such
that
\begin{eqnarray}
H = \left[
\begin{array}{c}
H_0\\
H_1\\
\vdots\\
H_{m}\\
\end{array}
\right],
\end{eqnarray}
where each $H_i$ has $n$ columns, obtaining the polynomial matrix
\begin{eqnarray}
G(D) =  {\tilde H}_0 + {\tilde H}_1 D + {\tilde H}_2 D^2 + \ldots +
{\tilde H}_{m} D^{m},
\end{eqnarray}
where the matrices ${\tilde H}_i$, for all $1\leq i\leq m$, are
derived from the respective matrices $H_i$ by adding zero-rows at
the bottom in such a way that the matrix ${\tilde H}_i$ has $\kappa$
rows in total, where $\kappa$ is the maximal number of rows among
the matrices $H_i$. The matrix $G(D)$ generates a convolutional code
with $\kappa$ rows and memory $m$. In this context, one has the
following result:

\begin{theorem}\cite[Theorem 3]{Aly:2007}\label{A}
Suppose that $C \subseteq { \mathbb F}_q^n$ is a linear code with
parameters ${[n, k, d]}_{q}$ and assume also that $H \in { \mathbb
F}_q^{(n-k)\times n}$ is a parity check matrix for $C$ partitioned
into submatrices $H_0, H_1, \ldots, H_m$ as in Eq.~(1) such that
$\kappa = \operatorname{rk}H_0$ and $\operatorname{rk}H_i \leq
\kappa$ for $1 \leq i\leq m$, where $\operatorname{rk}H_i$ denotes
the row rank of the matrix $H_i$. Consider the polynomial matrix
$G(D)$ as in Eq.~(2). Then the matrix $G(D)$ is a reduced basic
generator matrix of a convolutional code $V$. Additionally, if $d_f$
denotes the free distances of $V$ and $d^{\perp}$ is the minimum
distance of $C^{\perp}$, then $d_f \geq d^{\perp}$.
\end{theorem}


\section{Asymptotically Good Convolutional Codes}\label{III}

In this section we present the contributions of this paper. More
precisely, we construct new sequences of asymptotically good
convolutional codes.

First, we recall some basic concepts necessary for the development
of this paper. A subgroup $H$ of the symmetric group $S_n$ is called
transitive if for any pair $(i, j)$ with $i, j \in \{1, \ldots ,
n\}$, there exists a permutation $\rho \in H$ such that $\rho (i) =
j$. A permutation $\rho \in S_n$ is called an automorphism of the
code $C \subseteq {\mathbb F}_{n}^{q}$ if $(c_1, \ldots , c_n) \in C
\Rightarrow (c_{\rho(1)}, \ldots , c_{\rho (n)}) \in C$ holds for
all codewords $(c_1, \ldots , c_n) \in C$. The group
$\operatorname{Aut}(C)\subseteq S_n$ is the group of all
automorphisms of $C$. A code $C$ over ${ \mathbb F}_{q}$ of length
$n$ is called transitive if its automorphism group
$\operatorname{Aut}(C)$ is a transitive group of $S_n$. An example
of transitive codes are the well-known cyclic codes.

Recall that the generalized Singleton bound \cite{Rosenthal:2001} of
an $(n, k, \gamma ; m, d_{f} {)}_{q}$ convolutional code is given by
$ d_{f}\leq (n-k)[ \lfloor \gamma/k \rfloor +1 ] + \gamma +1$. We
put $s := (n-k)[ \lfloor \gamma/k \rfloor +1 ] + \gamma +1$. Based
on this upper bound for $d_{f}$, we can introduce more precisely the
concept of asymptotically good convolutional codes.

\begin{definition}\label{defmain1}
A sequence ${(V_{j})}_{j\geq 0}$ of convolutional codes with
parameters $(n_j, k_j,$ ${\gamma}_j ; m_j, {(d_f)}_{j} {)}_{q}$ is
said to be asymptotically good if ${\limsup}_{j \rightarrow \infty}$
$k_{j}/n_{j} > 0$ and ${\limsup}_{j \rightarrow \infty}
{(d_f)}_{j}/r_{j} > 0$ hold, where $s_j := (n_j -k_j)[ \lfloor
{\gamma}_j /k_j \rfloor +1 ] + {\gamma}_j +1$ and $r_j := \max \{n_j
, s_j \}$.
\end{definition}

\begin{remark}
Note that because both sequences of real numbers
$(k_{j}/n_{j})_{j\geq 0}$ and $({(d_f)}_{j}/r_j)_{j\geq 0}$ are
bounded, then there exist both ${\limsup}_{j \rightarrow \infty}
k_{j}/n_{j}$ and ${\limsup}_{j \rightarrow \infty}$
${(d_f)}_{j}/r_j$, and the definition makes sense.
\end{remark}

Let us recall the construction of a sequence of asymptotically good
transitive codes:

\begin{theorem}\cite[Theorem 1.5]{Henning:2005}\label{henni}
Let $q = l^2$ be a square. Then the class of transitive codes meets
the Tsfasman-Vladut-Zink bound. More precisely, let $R, \delta \geq
0$ be real numbers with $R = 1 - \delta -1/(l-1)$. Then there exists
a sequence ${(C_{j})}_{j\geq 0}$ of linear codes $C_j$ over ${
\mathbb F}_{q}$ with parameters ${[n_j , k_j , d_j ]}_{q}$ with the
following properties:\\
a) All $C_j$ are transitive codes.\\
b) $n_j \rightarrow \infty$ as $j \rightarrow \infty$.\\
c) ${\lim}_{j \rightarrow \infty} k_{j}/n_{j}\geq R $ and ${\lim}_{j
\rightarrow \infty} d_{j}/n_{j}\geq \delta $.
\end{theorem}

The codes shown in Theorem~\ref{henni} are AG codes constructed by
applying an asymptotically good tower of function fields over
${\mathbb F}_{q}$ (see \cite{Henning:2005}). For more details with
respect to AG codes, see \cite{Henning:2009}.

In this paper we only construct sequences of unit-memory
convolutional codes. Moreover, we construct reduced basic generator
matrices for the new asymptotically good sequences of convolutional
codes, i.e., the codes are non-catastrophic.

\begin{theorem}\label{main1}
Let $q = l^2$ be a prime power, where $l \geq 3$ is an integer. Then
there exists a sequence of asymptotically good convolutional codes,
over ${\mathbb F}_{q}$, derived from transitive codes.
\end{theorem}

\begin{proof}
We adopt the same notation of Theorem~\ref{henni}. Let $R > 0$ and
$\delta > 0$ be real numbers with $R = 1 - \delta -1/(l-1)$. Since
${\lim}_{j \rightarrow \infty} 1/ n_{j} = 0$ and ${\lim}_{j
\rightarrow \infty} k_{j}/n_{j}\geq R
> 0 $, then it follows that the sequence ${(k_{j})}_{j \geq 0}$ of positive
integers is not bounded. Taking $c=1$, there exists a nonnegative
integer $j_1$ such that $k_{j_1} > 1$. Analogously, there exists a
positive integer $j_2 > j_1$ such that $k_{j_2} > k_{j_1}$,
otherwise ${(k_{j})}_{j \geq 0}$ would be bounded. Assume by
induction that we have defined the subsequence for $n$ numbers $j_1
< j_2 < \ldots < j_n$, i.e., $1 < k_{j_1} < k_{j_2} < \ldots <
k_{j_n} $, where $j_1 < j_2 < \ldots < j_n$. Then there exists a
positive integer $j_{n+1} > j_{n}$ such that $k_{j_{n+1}} >
k_{j_n}$, otherwise ${(k_{j})}_{j \geq 0}$ would be bounded. Thus we
extract a subsequence ${(k_{j_{t}})}_{t \geq 0}$ of ${(k_{j})}_{j
\geq 0}$ with $k_{j_{t}}
> 1$ for each $t \geq 0$ and $k_{j_{t}} \rightarrow \infty$ as
$t \rightarrow \infty$. Moreover, it is clear that $n_{j_{t}}
\rightarrow \infty$ as $t \rightarrow \infty$ (alternatively,
because ${\lim}_{j \rightarrow \infty} k_{j}/n_{j} > 0$, and since
$n_j \rightarrow \infty$ when $j \rightarrow \infty$, there exists a
positive integer $n_0$ such that, $\forall j > n_0$, one has $k_j
> 1$).

We first consider the (Euclidean) dual $C_{j_{t}}^{\perp}$ of the
transitive code $C_{j_{t}}$ constructed in Theorem~\ref{henni}, for
every $t \geq 0$. To simplify the notation we put $C_{j_{t}}:=
C_{t}$, for all $t \geq 0$. A parity check matrix $G_{t}$ of
$C_{t}^{\perp}$ is a generator matrix of $C_t$, for all $t \geq 0$.
We construct a sequence of convolutional codes $V_t$, for all $t
\geq 0$, generated the by reduced basic matrix
$${ \mathbb G}_{t}(D) = G_{t}^{*} + {\tilde L}_{t} D,$$ where
$G_{t}^{*}$ is the submatrix of $G_{t}$ consisting of the first
$k_{t} - 1$ rows of $G_{t}$ and ${\tilde L}_{t}$ is the matrix
consisting of the last row of $G_{t}$ together $k_t - 2$ zero-rows
at the bottom. The code $V_t$ is a unit-memory code with parameters
$(n_{t}, k_{t}-1, 1 ; 1, {(d_{f})}_{t} {)}_{q}$, $t \geq 0$. From
Theorem~\ref{A}, one has ${(d_f)}_{t} \geq d_{t}$. It is obvious
that ${\lim}_{t \rightarrow \infty} (k_{t}-1)/n_{t} \geq R > 0 $. On
the other hand, since $s_t = n_t - k_t +2$ and because $k_t > 1$ for
all $t \geq 0$, then $r_t = n_t$; thus it follows that ${\lim}_{t
\rightarrow \infty} {(d_f)}_{t}/r_{t} > 0 $. Therefore, we have
constructed an asymptotically good sequence ${(V_{t})}_{t\geq 0}$ of
convolutional codes, as desired.
\end{proof}

\begin{theorem}\label{main1A}
Let $q = l^2$ be a prime power, where $l \geq 3$ is an integer.
Then, for any positive integer ${\gamma}_{0}> 1$ there exists a
sequence of asymptotically good convolutional codes, over ${\mathbb
F}_{q}$, with degree $\gamma = {\gamma}_{0}$.
\end{theorem}

\begin{proof}
By the same reasoning utilized in the first part of the proof of
Theorem~\ref{main1}, we construct a subsequence ${(k_{t})}_{t \geq
0}$ of ${(k_{j})}_{j \geq 0}$ such that $k_{t} \rightarrow \infty$
as $t \rightarrow \infty$ with ${\gamma}_{0} < k_{1} < k_{2} <
\ldots < k_{t} < \ldots $, where $n = 1, 2, 3, \ldots $.
Additionally, it is clear that $n_{t} \rightarrow \infty$ as $t
\rightarrow \infty$ (alternatively, proceeding similarly as in the
proof of Theorem~\ref{main1}, we show the existence of a positive
integer $n_0$ such that, for all $j > n_0$, one has $k_j >
{\gamma}_{0}$). We then take the dual $C_{t}^{\perp}$ of $C_{t}$,
for every $t \geq 0$ and consider a parity check matrix $G_{t}$ of
$C_{t}^{\perp}$ which is a generator matrix of $C_t$. We construct a
sequence of convolutional codes $V_t$, for all $t \geq 0$, generated
by the reduced basic matrix ${ \mathbb G}_{t}(D) = G_{t}^{*} +
{\tilde L}_{t} D$, where $G_{t}^{*}$ is the submatrix of $G_{t}$
consisting of the first $k_{t} - {\gamma}_{0}$ rows of $G_{t}$ and
${\tilde L}_{t}$ is the matrix consisting of the last ${\gamma}_{0}$
rows of $G_{t}$ together $k_t - 2{\gamma}_{0}$ zero-rows at the
bottom. The code $V_t$ is a unit-memory code with parameters
$(n_{t}, k_{t}-{\gamma}_{0}, {\gamma}_{0}; 1, {(d_{f})}_{t}
{)}_{q}$, $t \geq 0$. We know that $s_t = n_t - k_t + {\gamma}_{0} +
1$ and $r_t = n_t$. Since ${\lim}_{t \rightarrow \infty}
(k_{t}-{\gamma}_{0})/n_{t} > 0$ and ${\lim}_{t \rightarrow \infty}
{(d_f)}_{t}/r_{t} > 0 $ are true, the result follows.
\end{proof}


We next construct new families of asymptotically good convolutional
codes obtained by expanding, extending, puncturing, direct sum, the
$\langle u|u+v\rangle$ construction and by the product code
construction applied to the sequence of transitive codes shown in
Theorem~\ref{henni}. For more details concerning such techniques of
construction, see \cite{Macwilliams:1977,Huffman:2003}.

From now on, to simplify the notation, we always consider the
existence of a (sub)sequence ${(C_{j})}_{j\geq 0}$ of linear codes
with parameters ${[n_j , k_j , d_j ]}_{q}$, constructed in
Theorem~\ref{henni}, such that $k_j > 1$ (or $k_j >  {\gamma}_{0}$,
for all ${\gamma}_{0}\geq 1$, according to Theorem~\ref{main1A}) for
all $j \geq 0$.

\begin{theorem}\label{main2}
Let $q^{m} = l^2$ be a prime power, where $l \geq 3$ and $m > 1$ are
integers. Then there exists a sequence of asymptotically good
convolutional codes, over ${\mathbb F}_{q}$, derived from expansion
of transitive codes.
\end{theorem}

\begin{proof}
Let $R > 0$ and $\delta > 0$ be real numbers with $R = 1 - \delta
-1/(l-1)$. Consider a (sub)sequence ${(C_{j})}_{j\geq 0}$ of
asymptotically good linear codes $C_j$, over ${ \mathbb F}_{q^{m}}$,
with parameters ${[n_j , k_j , d_j ]}_{q^{m}}$, shown in
Theorem~\ref{henni}. Let $\beta=\{b_1, b_2, \ldots, b_{m}\}$ be a
basis of ${\mathbb F}_{q^{m}}$ over ${\mathbb F}_{q}$. We expand all
codes $C_j$ with respect to $\beta$ generating codes $\beta(C_{j})$,
over ${\mathbb F}_{q}$, with parameters ${[mn_j , mk_j , d_{j}^{*}
\geq d_j ]}_{q}$, for all $j \geq 0$. Consider the dual
${[\beta(C_{j})]}^{\perp}$ of the code $\beta(C_{j})$, for all $j
\geq 0$. A parity check matrix $G_{j}$ of ${[\beta(C_{j})]}^{\perp}$
is a generator matrix of $\beta(C_{j})$, for all $j \geq 0$.
Proceeding similarly as in the proof of Theorem~\ref{main1}, the
result follows.

Let $V_j$ be the convolutional code generated by the reduced basic
matrix ${ \mathbb G}_{j}(D) = G_{j}^{*} + {\tilde L}_{j} D$, where
$G_{j}^{*}$ is the submatrix of $G_{j}$ consisting of the $k_{j} -
1$ first rows of $G_{j}$ and ${\tilde L}_{j}$ is the matrix
consisting of the last row of $G_{j}$ together $k_j - 2$ zero-rows
at the bottom, for all $j \geq 0$. $V_j$ has parameters $(mn_{j},
mk_{j}-1, 1 ; 1, {(d_{f})}_{j} {)}_{q}$, where ${(d_f)}_{j} \geq
d_{j}^{*} \geq d_j$. We know that ${\lim}_{j \rightarrow \infty}
(mk_{j}-1)/mn_{j}\geq R
> 0 $. On the other hand, one has $s_j = m(n_j - k_j) + 3$ and $r_j \leq
mn_j + 1$. Thus, it follows that  ${(d_f)}_{j}/r_{j}\geq
{(d_f)}_{j}/(mn_j +1)\geq d_j / (mn_j +1)$. Because ${\lim}_{j
\rightarrow \infty} d_j / (mn_j +1) = [1/m] {\lim}_{j \rightarrow
\infty} d_j /n_j$, we conclude that ${\lim}_{j \rightarrow \infty}
{(d_f)}_{j}/ r_j > 0$. Thus, the sequence ${(V_{j})}_{j\geq 0}$ of
convolutional codes is asymptotically good and the proof is
complete.
\end{proof}

Applying the techniques of combining codes, we can get more
sequences of good convolutional codes.

\begin{theorem}\label{main3}
Let $q = l^2$ be a prime power, where $l \geq 3$ is an integer. Then
there exists a sequence of asymptotically good convolutional codes,
over ${\mathbb F}_{q}$, derived from direct sum of transitive codes.
\end{theorem}

\begin{proof}
Consider a (sub)sequence of codes ${(C_{j})}_{j\geq 0}$ shown in
Theorem~\ref{henni} with $k_j > 1$ for all $j \geq 0$. We construct
the sequence of direct sum codes $C_{j}\oplus C_{j}$ with parameters
${[2n_j , 2k_j , d_j ]}_{q}$, for all $j \geq 0$. Let $G_{j}$ be a
generator matrix of the code $C_{j}\oplus C_{j}$, for all $j \geq
0$. Consider the dual ${[C_{j}\oplus C_{j}]}^{\perp}$ of the code
$C_{j}\oplus C_{j}$, $j \geq 0$; $G_{j}$ is a parity check matrix
for ${[C_{j}\oplus C_{j}]}^{\perp}$.

Let $V_j$ be the convolutional code generated, by the reduced basic
matrix ${ \mathbb G}_{j}(D) = G_{j}^{*} + {\tilde L}_{j} D$, where
$G_{j}^{*}$ is the submatrix of $G_{j}$ consisting of the $k_{j} -
1$ first rows of $G_{j}$ and ${\tilde L}_{j}$ is the matrix
consisting of the last row of $G_{j}$ together $k_j - 2$ zero-rows
at the bottom, for all $j \geq 0$. $V_j$ has parameters $(2n_{j},
2k_{j}-1, 1 ; 1, {(d_{f})}_{j} {)}_{q}$, where ${(d_f)}_{j} \geq
d_j$. We know that $s_j = 2(n_j - k_j) + 3$ and $r_j \leq 2n_j + 1$,
for all $j \geq 0$, the free distance ${(d_f)}_{j}$ of the code
$V_j$ satisfies ${(d_f)}_{j}/r_{j}\geq d_j / (2n_j +1)$. Because
${\lim}_{j \rightarrow \infty} d_j / (2n_j +1) = [1/2] {\lim}_{j
\rightarrow \infty} d_j /n_j$, it follows that ${\lim}_{j
\rightarrow \infty} {(d_f)}_{j}/ r_j > 0$. On the other hand,
${\lim}_{j \rightarrow \infty} (2k_{j}-1)/2n_{j}= {\lim}_{j
\rightarrow \infty} k_{j}/n_{j}
> 0$. Therefore, the sequence ${(V_{j})}_{j\geq 0}$ of convolutional codes is
asymptotically good.
\end{proof}

\begin{theorem}\label{main4}
Let $q = l^2$ be a prime power, where $l \geq 3$ is an integer. Then
there exists a sequence of asymptotically good convolutional codes,
over ${\mathbb F}_{q}$, derived from the $\langle u|u+v\rangle$
construction of transitive codes.
\end{theorem}
\begin{proof}
The proof is similar to that of Theorem~\ref{main3}.
\end{proof}

In Theorems~\ref{main5} and \ref{main6}, we apply code extension and
puncturing of transitive codes, respectively, in order to construct
more two sequences of asymptotically good codes.

\begin{theorem}\label{main5}
Let $q = l^2$ be a prime power, where $l \geq 3$ is an integer. Then
there exists a sequence of asymptotically good convolutional codes,
over ${\mathbb F}_{q}$, derived from extension of transitive codes.
\end{theorem}

\begin{proof}
Assume that ${(C_{j})}_{j\geq 0}$ is a (sub)sequence of codes shown
in Theorem~\ref{henni} with $k_j > 1$ for all $ j > 0$. We construct
a sequence ${(C_{j}^{e})}_{j\geq 0}$ of codes with parameters ${[n_j
+ 1, k_j , d_{j}^{e} ]}_{q}$, where $d_{j}^{e}=d_j$ or
$d_{j}^{e}=d_j + 1$. The sequence of convolutional codes is
constructed similarly as in the proofs of Theorems~\ref{main1} to
\ref{main3}. We know that ${\lim}_{j \rightarrow \infty} (k_j -
1)/(n_{j}+1)= {\lim}_{j \rightarrow \infty} k_j/n_{j} > 0$ and that
$s_j = n_j - k_j + 4$. Hence, ${\lim}_{j \rightarrow \infty}
 {(d_f)}_{j}/r_j \geq {\lim}_{j \rightarrow \infty} d_j/r_{j} > 0$. The proof is complete.
\end{proof}

\begin{theorem}\label{main6}
Let $q = l^2$ be a prime power, where $l \geq 3$ is an integer. Then
there exists a sequence of asymptotically good convolutional codes,
over ${\mathbb F}_{q}$, derived from puncturing transitive codes.
\end{theorem}

\begin{proof}
Let ${(C_{j})}_{j\geq 0}$ be a (sub)sequence of codes shown in
Theorem~\ref{henni}, with $k_j > 1$ for all $ j > 0$. By puncturing
the codes $C_{j}$, $j\geq 0$, in the $i$th coordinate, we construct
a new sequence ${(C_{j}^{p})}_{j\geq 0}$ of codes with parameters
${[n_j -1 , k_j , d_{j}^{p}]}_{q}$, where $d_{j}^{p}= d_j -1$ if
$C_j$ has a minimum weight codeword with a nonzero $i$th coordinate
and $d_{j}^{p}= d_j$ otherwise. Note that we can assume w.l.o.g.
that $d_j > 1$ for all $j\geq 0$; see the first part of the proof of
Theorem~\ref{main1} applied to $d_j$ instead of $k_j$. We construct
a sequence of convolutional codes as in Theorems~\ref{main1} to
\ref{main3}. The sequence ${(V_{j})}_{j\geq 0}$ consists of
convolutional codes with parameters $(n_j - 1, k_j -1, 1 ; 1,
{(d_{f})}_{j} {)}_{q}$, where ${(d_f)}_{j} \geq d_j - 1$ and $s_j =
n_j - k_j + 2$. It is clear that ${\lim}_{j \rightarrow \infty} (k_j
- 1)/(n_{j} - 1) > 0$ and ${\lim}_{j \rightarrow \infty}
{(d_f)}_{j}/r_{j} > 0$. Therefore, ${(V_{j})}_{j\geq 0}$ is
asymptotically good and we are done.
\end{proof}

Before showing Theorem~\ref{main7}, we must recall the direct
product construction of linear (block) codes. For more details, see
\cite{Macwilliams:1977}.

\begin{definition}\cite{Macwilliams:1977}\label{prodcode}
Let $C_1={[n_1, k_1, d_1]}_{q}$ and $C_2={[n_2, k_2, d_2]}_{q}$ be
linear codes both over ${\mathbb F}_{q}$, with generator matrices
$G^{(1)}$ and $G^{(2)}$, respectively. Then the product code
$C_{\pi}:= C_1 \otimes C_2$ is a linear code, over ${\mathbb
F}_{q}$, with parameters $C_{\pi}={[n_1 n_2, k_1 k_2, d_1
d_2]}_{q}$, generated by the matrix $G^{(\pi)}=G^{(1)} \otimes
G^{(2)}$, where $\otimes$ denotes the Kronecker product of matrices,
that is,
\begin{eqnarray*}
G^{(\pi)} = \left[
\begin{array}{cccc}
g_{11}^{(1)}G^{(2)} & g_{12}^{(1)}G^{(2)} & \cdots & g_{1 n_1}^{(1)}G^{(2)}\\
g_{21}^{(1)}G^{(2)} & g_{22}^{(1)}G^{(2)} & \cdots & g_{2 n_1}^{(1)}G^{(2)}\\
\vdots & \vdots & \vdots & \vdots\\
g_{k_1 1}^{(1)}G^{(2)} & g_{k_1 2}^{(1)}G^{(2)} & \cdots & g_{k_1 n_1}^{(1)}G^{(2)}\\
\end{array}
\right],
\end{eqnarray*}
where $g_{ij}^{(t)}$, $t=1, 2$, is the $ij$-th entry of the matrix
$G^{(t)}$, respectively.
\end{definition}

As we will see in Theorem~\ref{main7} given in the following, direct
product codes obtained from transitive codes are also asymptotically
good.

\begin{theorem}\label{main7}
Let $q = l^2$ be a prime power, where $l \geq 3$ is an integer. Then
there exists a sequence of asymptotically good convolutional codes,
over ${\mathbb F}_{q}$, derived from direct product of transitive
codes.
\end{theorem}

\begin{proof}
Let ${(C_{j})}_{j\geq 0}$ be a (sub)sequence of codes shown in
Theorem~\ref{henni}, with $k_j > 1$ for all $ j > 0$. Let us
consider the sequence of product codes ${(C_{j}\otimes
C_{j})}_{j\geq 0}$. This sequence consists of codes with parameters
${[{(n_j)}^{2}, {(k_j)}^{2} , {(d_j)}^{2}]}_{q}$. Constructing the
sequence ${(V_{j})}_{j\geq 0}$ of convolutional codes similarly as
in the proof of Theorems~\ref{main1} to \ref{main6}, we obtain codes
with parameters ${({(n_j)}^{2}, {(k_j)}^{2} - 1, 1 ; 1,
{(d_{f})}_{j} )}_{q}$, with ${(d_f)}_{j} \geq {(d_j)}^{2}$, where
$s_j = {(n_j)}^{2} - {(k_j)}^{2} + 3$. It is straightforward to see
that ${\lim}_{j \rightarrow \infty} [{(k_j)}^{2} - 1]/{(n_j)}^{2} >
0$ and ${\lim}_{j \rightarrow \infty} {(d_f)}_{j}/r_{j} > 0$ are
true. Thus, the result follows.
\end{proof}

Let $C$ be a linear code. Recall that $C$ is called
\emph{self-orthogonal} if $C \subset C^{\perp}$, and $C$ is called
\emph{self-dual} if $C = C^{\perp}$, where $C^{\perp}$ is the
(Euclidean) dual of the code $C$. The following theorem will be also
utilized to construct more sequences of asymptotically good
convolutional codes:

\begin{theorem}\cite[Theorem 1.6]{Henning:2005}\label{henni1}
Let $q = l^2$ be a square. Then the class of self-orthogonal codes
and the class of self-dual codes meets the Tsfasman-Vladut-Zink
bound. More precisely, the following are true:
\begin{enumerate}
\item[ (i)] Let $0\leq R \leq 1/2$ and $\delta \geq 0$ be real numbers with $R =
1 - \delta -1/(l-1)$. Then there exists a sequence ${(C_{j})}_{j\geq
0}$ of linear codes $C_j$ over ${ \mathbb F}_{q}$ with parameters
${[n_j , k_j , d_j ]}_{q}$ such that:\\
a) All $C_j$ are self-orthogonal codes.\\
b) $n_j \rightarrow \infty$ as $j \rightarrow \infty$.\\
c) ${\lim}_{j \rightarrow \infty} k_{j}/n_{j}\geq R $ and ${\lim}_{j
\rightarrow \infty} d_{j}/n_{j}\geq \delta $.

\item[ (ii)] There exists a sequence ${(D_{j})}_{j\geq
0}$ of self-dual codes $D_j$ over ${ \mathbb F}_{q}$ with parameters
${[n_j , n_j /2 , d_j ]}_{q}$ such that $n_j \rightarrow \infty$ and
${\lim}_{j \rightarrow \infty} d_{j}/n_{j}\geq 1/2 -1/(l-1)$.
\end{enumerate}
\end{theorem}

\begin{theorem}\label{main8}
Let $q = l^2$ be a prime power, where $l \geq 3$ is an integer.
Consider the two sequences ${(C_{j})}_{j\geq 0}$ and
${(D_{j})}_{j\geq 0}$ of self-orthogonal and self-dual linear codes
respectively, over ${ \mathbb F}_{q}$ shown in Theorem~\ref{henni1}.
Then the following hold
\begin{itemize}
\item [ (i)] There exist two sequences of asymptotically good
convolutional codes ${(V_{j})}_{j\geq 0}$ and ${(W_{j})}_{j\geq 0}$,
derived respectively from ${(C_{j})}_{j\geq 0}$ and
${(D_{j})}_{j\geq 0}$.

\item [ (ii)] There exist sequences of asymptotically good
convolutional codes ${(V_{j}^{k})}_{j\geq 0}$ derived, respectively,
from ${(C_{j}^{k})}_{j\geq 0}$, $k=1, \ldots 6$, where the codes
$C_{j}^{1}, C_{j}^{2},$ $\ldots , C_{j}^{6}$ are obtained from
${(C_{j})}_{j\geq 0}$ by expanding, direct sum, the $\langle
u|u+v\rangle$ construction, extending, puncturing and from direct
product construction, respectively.

\item [ (iii)] There exist sequences of asymptotically good
convolutional codes ${(W_{j}^{k})}_{j\geq 0}$ derived, respectively,
from ${(D_{j}^{k})}_{j\geq 0}$, $k=1, \ldots 6$, where the codes
$D_{j}^{1}, D_{j}^{2},$ $\ldots , D_{j}^{6}$ are obtained from
${(D_{j})}_{j\geq 0}$ by expanding, direct sum, the $\langle
u|u+v\rangle$ construction, extending, puncturing and from direct
product construction, respectively.
\end{itemize}
\end{theorem}
\begin{proof}
The proof is similar to that proofs of Theorems~\ref{main1} to
\ref{main7}, respectively.
\end{proof}

Let us recall the definition of asymptotically good block codes:

\begin{definition}\label{asg}
A sequence ${(C_{j})}_{j\geq 0}$ of linear codes over ${ \mathbb
F}_{q}$ with parameters $[n_j ,$ $k_j , d_j {]}_{q}$ is said to be
asymptotically good if $n_j \rightarrow \infty$ as $j \rightarrow
\infty$ and if ${\limsup}_{j \rightarrow \infty}$ $k_{j}/n_{j}> 0 $
and ${\limsup}_{j \rightarrow \infty} d_{j}/n_{j} > 0$ hold.
\end{definition}

Theorem~\ref{gene}, given in the following, establishes that the
results presented in this paper are also true in a more general
scenario, when considering arbitrary families of asymptotically good
linear block codes in order to construct AGCC's:

\begin{theorem}\label{gene}
Let $q$ be a prime power. Assume that there exists a sequence
${(C_{j})}_{j\geq 0}$ of asymptotically good linear (block) codes
$C_j$, where both limits ${\lim}_{j \rightarrow \infty}$
$k_{j}/n_{j}$ and ${\lim}_{j \rightarrow \infty} d_j /n_{j}$ exist.
Then the following hold:
\begin{itemize}
\item [ (i)] There exists a sequence of asymptotically good
convolutional codes ${(V_{j})}_{j\geq 0}$ derived from
${(C_{j})}_{j\geq 0}$.

\item [ (ii)] There exist sequences of asymptotically good
convolutional codes ${(V_{j}^{k})}_{j\geq 0}$,  $k=1, \ldots 6$,
derived, respectively, from the sequences ${(C_{j}^{k})}_{j\geq 0}$
of linear codes, $k=1, \ldots 6$, where the codes $C_{j}^{1},
C_{j}^{2}, \ldots , C_{j}^{6}$ are obtained from ${(C_{j})}_{j\geq
0}$ by expanding, direct sum, the $\langle u|u+v\rangle$
construction, extending, puncturing and from direct product
construction, respectively.
\end{itemize}
\end{theorem}

\begin{proof}
The proof is similar to that proofs of Theorems~\ref{main1} to
\ref{main7}, respectively.
\end{proof}


\section{Summary}\label{IV}

We have constructed new sequences of asymptotically good
convolutional codes derived from sequences of transitive,
self-orthogonal and self-dual block codes that attain the
Tsfasman-Vladut-Zink bound. Furthermore, more sequences of new
asymptotically good convolutional codes were obtained by applying
the techniques of expanding, extending, puncturing, direct sum, the
$\langle u | u + v\rangle$ construction and the product code
construction to these block codes. Additionally, we have shown that
our new constructions of AGCC's also hold when applied to all
sequences of asymptotically good block codes where ${\lim}_{j
\rightarrow \infty} k_{j}/n_{j}$ and ${\lim}_{j \rightarrow \infty}
d_j /n_{j}$ exist.

\section*{Acknowledgment}
I would like to thank the anonymous referee for his/her valuable
comments and suggestions that improve significantly the quality and
the presentation of this paper. This research has been partially
supported by the Brazilian Agencies CAPES and CNPq.


\small

\begin{thebibliography}{10}

\bibitem{Aly:2007}
S.A. Aly, M. Grassl, A. Klappenecker, M. R\"otteler, P.K.
Sarvepalli.
\newblock Quantum convolutional BCH codes.
\newblock In {\em Proc. Canadian Workshop on Information Theory
(CWIT)}, pp.180--183, 2007.

\bibitem{Bocharova:2016}
I.E. Bocharova, B.D. Kudryashov and R. Johannesson.
\newblock Searching for binary and nonbinary block and convolutional LDPC
codes.
\newblock {\em IEEE Trans. Inform. Theory}, 62(1):163--183, 2016.

\bibitem{Chen:2005}
Z.G. Chen and S. Bates.
\newblock Construction of low-density parity-check
convolutional codes through progressive edge-growth.
\newblock {\em IEEE Commun Lett}, 9:1058–-1060, 2005.

\bibitem{Felstrom:1999}
A.J. Felstrom and K.Sh. Zigangirov.
\newblock Time-varying periodic convolutional codes with low-density
parity-check matrices.
\newblock {\em IEEE Trans. Inform. Theory}, 45(6):2181--2191, 1999.

\bibitem{Forney:1970}
G.D. Forney Jr.
\newblock Convolutional codes I: algebraic structure.
\newblock {\em IEEE Trans. Inform. Theory}, 16(6):720--738, 1970.

\bibitem{Henning:1995}
A. Garcia, H. Stichtenoth.
\newblock A tower of Artin-Schreier extensions of function fields attaining
the Drinfeld-Vladut¸ bound.
\newblock {\em Invent. Math.}, 121:211--222, 1995.

\bibitem{Henning:1996}
A. Garcia, H. Stichtenoth.
\newblock On the asymptotic behaviour of some towers of function fields over
finite fields.
\newblock {\em J. Number Theory}, 61:248--273, 1996.

\bibitem{Gluesing:2006}
H. Gluesing-Luerssen, J. Rosenthal, R. Smarandache.
\newblock Strongly MDS convolutional codes.
\newblock {\em IEEE Trans. Inform. Theory}, 52:584--598, 2006.

\bibitem{Schmale:2006}
H. Gluesing-Luerssen, W. Schmale.
\newblock Distance bounds for convolutional codes and some optimal codes.
\newblock e-print arXiv:math/0305135.


\bibitem{Goppa:1972}
V.D. Goppa.
\newblock A new class of linear correcting codes.
\newblock {\em Probl. Peredachi Inf.}, (6):24–-30, 1970.

\bibitem{Goppa:1981}
V.D. Goppa.
\newblock Codes on algebraic curves.
\newblock {\em Soviet Math. Dokl}, 22(1):170--172, 1981.


\bibitem{Huffman:2003}
W.C. Huffman, V. Pless.
\newblock {\em Fundamentals of Error Correcting Codes}.
\newblock Cambridge Univ. Press, 2003.

\bibitem{Hutchinson:2005}
R. Hutchinson, J. Rosenthal, R. Smarandache.
\newblock Convolutional codes with maximum distance profile.
\newblock {\em Systems and Control Letters}, 54(1):53--63, 2005.

\bibitem{Iglesias:2009}
J.I. Iglesias-Curto.
\newblock Generalized AG convolutional codes.
\newblock {\em Advances in Mathematics of Communications}, 3(4):317--328,
2009.


\bibitem{Johannesson:1999}
R. Johannesson, K.S. Zigangirov.
\newblock {\em Fundamentals of Convolutional Coding.}
\newblock Digital and Mobile Communication, Wiley-IEEE Press, 1999.

\bibitem{Justesen:1973}
J. Justesen.
\newblock New convolutional code constructions and a class of asymptotically
good time-varying codes.
\newblock {\em IEEE Trans. Inform. Theory}, 19(2):220--225, 1973.

\bibitem{Kuijper:2009}
M. Kuijper, R. Pinto.
\newblock On Minimality of convolutional ring encoders.
\newblock {\em IEEE Trans. Inform. Theory}, 55(11):4890--4897, 2009.


\bibitem{LaGuardia:2014}
G.G. La Guardia.
\newblock On classical and quantum MDS-convolutional BCH codes.
\newblock {\em IEEE Trans. Inform. Theory}, 60(1):304--312, 2014.

\bibitem{LaGuardia:2014A}
G.G. La Guardia.
\newblock On negacyclic MDS-convolutional codes.
\newblock {\em Linear Alg. Applic.}, 448:85–-96, 2014.

\bibitem{LaGuardia:2016}
G.G. La Guardia.
\newblock On optimal constacyclic codes.
\newblock {\em Linear Alg. Applic.}, 496:594--610, 2016.

\bibitem{Costello:2010}
M. Lentmaiert, D.G.M. Mitchell, G. Fettweist, D.J. Costello, Jr.
\newblock Asymptotically good LDPC convolutional codes with AWGN channel
thresholds close to the Shannon limit.
\newblock In {\em 6th Int. Symp. on Turbo Codes and Iterative Inform.
Processing}, pp:324--328, 2010.

\bibitem{MacKay:2003}
D.J.C. MacKay, M.S. Postol.
\newblock Weaknesses of Margulis and Ramanujan-Margulis
low-density parity-check codes.
\newblock {\em Electron. Notes Theoretical Comput. Sci.}, 74:97-–104, 2003.

\bibitem{Macwilliams:1977}
F.J. MacWilliams, N.J.A. Sloane.
\newblock {\em The Theory of Error-Correcting Codes}.
\newblock North-Holland, 1977.

\bibitem{Costello:2008}
D.G.M. Mitchell, A.E. Pusane, K.Sh. Zigangirov, D.J. Costello Jr..
\newblock Asymptotically good LDPC convolutional codes based on
protographs.
\newblock In {\em Proc. Int. Symp. Inform. Theory (ISIT)}, pp.1030--1034, 2008.

\bibitem{Mitchell:2015}
D.G.M. Mitchell, M. Lentmaier, D.J. Costello Jr..
\newblock Spatially coupled LDPC codes constructed from protographs.
\newblock {\em IEEE Trans. Inform. Theory}, 61(9):4866--4889, 2015.

\bibitem{Mitchell:2013}
D.G.M. Mitchell, A.E. Pusane, D.J. Costello Jr..
\newblock Minimum distance and trapping set analysis of protograph-based LDPC
convolutional codes.
\newblock {\em IEEE Trans. Inform. Theory}, 59(1):254--281, 2013.

\bibitem{Liu:2014}
L. Mu, X. Liu and C. Liang.
\newblock Improved construction of LDPC convolutional codes with semi-random
parity-check matrices.
\newblock {\em Science China}, 57:022304:1–-022304:10, 2014.

\bibitem{Napp:2017}
D. Napp, R. Pinto, M. Toste.
\newblock On MDS convolutional codes over ${\mathbb
Z}_{p}^{r}$.
\newblock {\em Designs, Codes and Cryptography}, 83(1):101–-114, 2017.

\bibitem{Papadimitriou:2001}
P.D. Papadimitriou, C.N. Georghiades.
\newblock On asymptotically optimum rate $1/n$ convolutional codes for a given
constraint length.
\newblock {\em IEEE Commun. Letters}, 5(1):25--27, 2001.

\bibitem{Piret:1972}
Ph Piret.
\newblock {\em Phillips Res. Report}, (27):257--271, 1972.

\bibitem{Piret:1988}
Ph. Piret.
\newblock {\em Convolutional Codes: An Algebraic Approach.}
\newblock Cambridge, Massachusetts: The MIT Press, 1988.

\bibitem{Rosenthal:1999}
J. Rosenthal, R. Smarandache.
\newblock Maximum distance separable convolutional codes.
\newblock {\em Applicable Algebra in Eng. Comm. Comput.}, 10:15--32, 1999.

\bibitem{Rosenthal:2001}
R. Smarandache, H. G.-Luerssen, J. Rosenthal.
\newblock Constructions of MDS-convolutional codes.
\newblock {\em IEEE Trans. Inform. Theory}, 47(5):2045--2049, 2001.

\bibitem{Sridharan:2007}
A. Sridharan, D. Truhachev, M. Lentmaier, D.J. Costello Jr., K.Sh. Zigangirov.
\newblock Distance bounds for an ensemble of LDPC convolutional codes.
\newblock {\em IEEE Trans. Inform. Theory}, 53(12):4537--4555, 2007.

\bibitem{Henning:2005}
H. Stichtenoth.
\newblock Transitive and self-dual codes attaining the Tsfasman-Vladut-Zink
bound.
\newblock {\em IEEE Trans. Inform. Theory}, 52(5):2218--2224, 2006.

\bibitem{Henning:2009}
H. Stichtenoth.
\newblock {\em Algebraic Function Fields and Codes}.
\newblock Springer, 2009.

\bibitem{Uchikawa:2011}
H. Uchikawa, K. Kasai and K. Sakaniwa.
\newblock Design and performance of rate-compatible non-binary LDPC
convolutional codes.
\newblock {\em IEICE Trans. on Fundamentals of Electronics, Communications and Computer Sciences},
E94.A(11):2135--2143, 2011.


\end{thebibliography}
\end{document}